\documentclass[11pt, reqno]{amsart}

\usepackage{amsfonts, amssymb, amsmath, amsthm, amsaddr, enumerate, multicol, xspace, array, multirow, graphicx, lipsum, bbm, mathrsfs, mathabx, tikz-cd, wasysym, mathtools, textcomp, etoolbox, comment}
\usepackage{CJKutf8}

\usepackage[margin=1.5in]{geometry}
\usepackage{hyperref}
\hypersetup{
	colorlinks=true,
	linkcolor=blue,
	filecolor=magenta,      
	urlcolor=cyan,
	citecolor=blue
}
\theoremstyle{plain}
\newtheorem{theorem}{Theorem}

\newtheorem{lemma}[theorem]{Lemma}
\newtheorem{proposition}[theorem]{Proposition}
\newtheorem{corollary}[theorem]{Corollary}

\newtheorem{conjecture}[theorem]{Conjecture}

\theoremstyle{definition}
\newtheorem{definition}[theorem]{Definition}
\newtheorem*{remark}{Remark}
\newtheorem*{example}{Example}

\numberwithin{theorem}{section} 

\theoremstyle{plain}

\numberwithin{equation}{section}

\newcommand{\N}{\mathbb{N}}
\newcommand{\Z}{\mathbb{Z}}
\newcommand{\Q}{\mathbb{Q}}

\newcommand{\C}{\mathbb{C}}
\newcommand{\F}{\mathbb{F}}

\newcommand{\roi}{{\mathcal O}}

\renewcommand{\phi}{\varphi}

\newcommand{\B}{\mathcal{B}}

\newcommand{\va}{\mathbf{a}}
\newcommand{\vb}{\mathbf{b}}

\newcommand{\vA}{\mathbf{A}}
\newcommand{\vB}{\mathbf{B}}

\newcommand{\vzero}{\mathbf{0}}

\DeclareMathOperator{\rado}{Rado}

\DeclareMathOperator{\RP}{RP}

\providecommand{\abs}[1]{\left\vert#1\right\vert}

\tikzset{commutative diagrams/.cd,
	smoffset/.style={start anchor=center,end anchor=center,draw=none}
}

\title{A Multicolored Variant of Rado's Theorem in Ramsey Theory}

\author{Hongyi Zhou}

\address{Department of Mathematical Sciences, Carnegie Mellon University\\Pittsburgh, PA 15213}

\email{hongyizh@andrew.cmu.edu}

\begin{document}

	\begin{abstract}
		A classical question in combinatorial number theory asks  whether an equation has a solution inside a particular subset of its domain. The Rado's Theorem gives a set of necessary and sufficient conditions for a systems of linear equations to have a monochromatic solution whenever the positive integers are finitely colored.
		In this paper, we provide a variant of this theorem. For $k \ge 2$, we present conditions such that, when the set of variables is partitioned into $k$ subsets, there is a solution such that the variables of each subset are monochromatic, which we call a semi-monochromatic solution. We adopt the smod $p$ coloring from \cite{grs90} but turn the existence of semi-monochromatic solution into the existence of common roots of linear polynomials. With this idea, one can further generalize the theorem to systems of linear equations over general algebraic number fields.
	\end{abstract}


	\maketitle

	\section{Introduction}

	\subsection{Background}
	
	Ramsey theory is a branch of combinatorics which studies properties of partitions and subsets of a set of items (such as numbers, graphs, etc.). To help describe the partitions, we say an $r$-coloring of elements in a set $S$ is a function $\chi:S \to C$, where $C$ is the set of colors. We say a subset $T \subseteq S$ is \emph{monochromatic} if $\chi(T) = \{c_\ast\}$ for some $c_\ast \in C$. Typical problems in Ramsey theory ask about the existence of monochromatic substructures when a sufficiently large structure is finitely colored.

	One result in Ramsey Theory is Schur's Theorem (\cite{grs90}, Chapter 3, Theorem 1), which states that there exist three positive integers $x, y, z$ such that $x + y = z$, whenever $\Z_+$ is finitely colored. Another result is Van der Waerden's Theorem (\cite{grs90}, Chapter 2, Theorem 1), which shows the existence of monochromatic arithmetic progress in $\Z_+$ of arbitrary length, whenever $\Z_+$ is finitely colored.
	
	Rado's Theorem generalizes the above theorems. It shows that a system has monochromatic solutions in $\Z_+$ whenver $\Z_+$ is finitely colored if and only if it satisfies some conditions. To state the conditions and the theorem, we first give two definitions.
	
	\begin{definition}
		\label{def:regular}
		Let $S$ be a system of equations in variables $x_1, \dots, x_n$, with coefficients in $\Z$. We say that $S$ is \emph{$r$-regular} if, given any $r$-coloring of $\Z_+$, there exists monochromatic $z_1, \dots, z_n \in \Z_+$ that is a solution to $S$. We say that $S$ is \emph{regular} if it is $r$-regular for all $r \in \Z_+$.
	\end{definition}
	
	\begin{definition}
		\label{col}
		Suppose $A = (\va_i^{(j)})$ is a $d \times n$-matrix, where $\va_1, \dots, \va_n$ are the column vectors. We say $A$ satisfies the \emph{columns condition} if:
		
		There is a partition of $[n] = \bigcup_{s=1}^t I_s$ such that, for each $s = 1, \dots, t$, the vector $\vA_s = \sum_{i \in I_s} \va_i$ can be expressed as a $\Z$-linear combination of the vectors in
		\begin{equation}
			\label{set-lincombo}
			\bigcup_{s' < s} \{\va_i: i \in I_{s'}\}.
		\end{equation}
	\end{definition}
	\begin{remark}
		In the above definition, when $s = 1$, as the set (\ref{set-lincombo}) is empty, the condition reduces to $\vA_1 = \vzero$.
	\end{remark}
	
	Then we may state the theorem:
	
	\begin{theorem}\emph{(Rado's Theorem. \cite{grs90}, Chapter 4, Theorem 5)}
		\label{rado}
		Suppose $S$ is a system of linear equations on variables $x_1, \cdots, x_n$ and with coefficients in $\Z$. Write $S$ explicitly as
		$$\sum_{i=1}^{n} \va_i x_i = 0,$$
		with coefficient matrix $A = (\va_i^{(j)})$.
		The system $S$ is regular iff $A$ satisfies the columns conditions.
	\end{theorem}

	If we require only a subset of variables in the solution being monochromatic, it is natural to expect weaker conditions on the coefficient matrix would suffice. This is indeed so, and it is the main result of our paper. From here we should introduce more definitions to formally present the conditions.

	\vspace{0.3cm}

	We endow a system of linear equations with additional information.
	
	\begin{definition}
		\label{dk-system}
		A \emph{$D$-dimensional $k$-partite linear system} on variables $\{x_{j,i}:  j \in [k],  i \in [N_j]\}$, where $N_j, k \in \Z_+$, is a system of linear equations of the form:
		\begin{equation}
			\label{kpar}
			\sum_{j = 1}^{k} \sum_{i = 1}^{N_j} \va_{j, i} x_{j, i} = \vzero,
		\end{equation}
		where every $\va_{j,i}$ is a vector in $\Z^D$. We call this a \emph{(D, k)-system} for simplicity.
		
		For convenience, we write a \emph{bipartite linear system} in the form
		\begin{equation}
			\label{2par}
			\sum_{i=1}^{N} \va_i x_i + \sum_{j=1}^{M} \vb_j y_j = \vzero.
		\end{equation}
	\end{definition}

	\begin{definition}
		Let $S$ be a $(D,k)$-system on variables $\{x_{j,i}: j \in [k],  i \in [N_j]\}$ with integer coefficients of the form (\ref{kpar}).
		Suppose $\Z$ is finitely colored. A \emph{semi-chromatic solution} in $\Z$ to $S$ is an assignment to the variables  $$x_{j,i} = z_{j,i}, \hspace{1cm} i \in [N_j], j \in [k]$$ such that for each $j$, the set $\{z_{j, i}: i \in N_j\}$ is monochromatic. In this paper, we pay attention to only the nontrivial solutions, i.e., not all zeros.
	\end{definition}
	
	\begin{example}
		Consider the system of equations:
		\begin{equation}
			\label{exm:k=2semi}
			\begin{cases}
				2x_1 + 2x_2 - 5y_1 + 10 y_2 = 0 \\
				x_1 + 3x_2 + 7y_1 - 2y_2 = 0
			\end{cases}.
		\end{equation}
		We color $z \in \Z$ by red if $3 \mid z$ and blue otherwise. This is a valid 2-coloring of $\Z$.
		
		The equation (\ref{exm:k=2semi}) admits a solution $(x_1, x_2, y_1, y_2) = (6, -6, 2, 1)$, where $\chi(6) = \chi(-6) = \text{red}$ and $\chi(2) = \chi(1) = \text{blue}$. So this solution is semi-monochromatic for this (2, 2)-system.
	\end{example}
	
	\begin{definition}
		Let $S$ be a $(D,k)$-system. We say $S$ is \emph{semi-regular} if whenever $\Z$ is finitely colored, $S$ admits a nontrivial (not-all-zero) semi-monochromatic solution in $\Z$.
	\end{definition}
	\begin{remark}
		Here, unlike the notion of a regular system as in Definition \ref{def:regular}, we consider colorings on $\Z$ instead of $\Z_+$. There is an equivalent version of the Rado's Theorem presented in \cite{rado33}, \cite{canonical}, and Chapter 5.4 of \cite{canonical}, which considers colorings on $\Z \setminus \{0\}$. We may move one step further by considering colorings on $\Z$ while discarding the trivial (all-zero) solution as a (semi-)monochromatic solution. Then our main result is a generalization of the Rado's Theorem in this sense.
	\end{remark}
	
	To study the solution regularity of a $(D, k)$-system, we devise an extended version of the columns conditions.
	
	\begin{definition}
		\label{kcol}
		Let $S$ be a $(D,k)$-system of the form (\ref{kpar}).
		The \emph{$k$-columns condition} is that we can partition each set of the variables into $t$ disjoint subsets, namely for each $j \in [k]$, $[N_j] = \bigsqcup_{s=1}^{t} I_{j,s}$, and there exists not-all-zero integer weights $\delta_1, \dots, \delta_k$, such that, when writing $\vA_{j, s} = \sum_{i \in I^j_s} \va_{j,i}$, we have
		\begin{itemize}
			\item $\sum_{j=1}^k \delta_j \vA_{j, 1} = \vzero$;
			
			\item for each $1 < s \le t$, $\sum_{j=1}^{k} \delta_j \vA_{j, s}$ is a linear combination (over $\Z$) of the set of vectors
			$$\bigcup_{s' < s} \bigcup_{j=1}^k \{\va_{j,i}: i \in I_{j,s'}\}.$$
		\end{itemize}
	\end{definition}
	
	\begin{example}
		Consider still the system (\ref{exm:k=2semi}). Write it in the vector form:
		\begin{equation*}
			\begin{bmatrix}
				2 \\ 1
			\end{bmatrix} x_1 +
			\begin{bmatrix}
				2 \\ 3
			\end{bmatrix} x_2 +
			\begin{bmatrix}
				-5 \\ 7
			\end{bmatrix} y_1 + 
			\begin{bmatrix}
				10 \\ -2
			\end{bmatrix} y_2 = 0.
		\end{equation*}
		Observe that
		\begin{equation*}
			\begin{bmatrix}
				2 \\ 1
			\end{bmatrix} + 
			\begin{bmatrix}
				2 \\ 3
			\end{bmatrix} = 
			\begin{bmatrix}
				4 \\ 4
			\end{bmatrix}, \hspace{1cm}
			\begin{bmatrix}
				-5 \\ 7
			\end{bmatrix} + 
			\begin{bmatrix}
				10 \\ -2
			\end{bmatrix} = 
			\begin{bmatrix}
				5 \\ 5
			\end{bmatrix},
		\end{equation*}
		so taking $\delta_1 = 5$ and $\delta_2 = -4$ shows that the system satisfies the 2-columns condition.
	\end{example}

	\subsection{Main Results}
	
	Now we provide the variant of Rado's Theorem with our $k$-columns conditions.

	\begin{theorem}
		\label{kgeneral}
		View $S$ as a $(D,k)$-system. It is semi-regular if and only if it satisfies the $k$-columns condition.
	\end{theorem}

	The proof to the ``if'' part can be reduced to the original Rado's Theorem, while the ``only if'' part requires some strategies in number theory and linear algebra, which we will display later. The above example is a consequence of the following corollary:
	
	\begin{corollary}
		\label{1dim}
		Suppose $S$ is a bipartite system of the form (\ref{2par}) on variables $x_1, \dots, x_N$ and $y_1, \dots, y_M$, with integer coefficients. It is semi-regular if and only if it satisfies the 2-columns condition.
	\end{corollary}

	\section{Technical Lemmas}

	This section gives the technical preliminaries in linear algebra and algebraic geometry. These results are fundamental to our polynomial methods.

	\subsection{Primes to Exclude} The following lemmas are consequences of a lemma from \cite{grs90}.
	
	\begin{lemma}\emph{(\cite{grs90}, Lemma 6, pg. 74)}
		\label{finiteprimes}
		Let $\vA$ and $\{\va_i: i \in I\}$ be $d$-dimensional vectors in $\Z^D$, where $I$ is a finite index set. Suppose that $\vA$ is not in the $\Q$-vector space generated by the $\va_i$'s. Then, for all but a finite number of primes $p$, $\vA$ cannot be expressed as a linear combination of the $\va_i$'s modulo $p$. Moreover, $p^m \vA$ cannot be expressed as a linear combination of the $\va_i$'s modulo $p^{m+1}$ for any $m \ge 0$.
	\end{lemma}

	Notice that linear dependence of vectors is equivalent to the vanishing of determinants of all maximal submatrices, when we put the vectors as columns of a matrix. So we have the following auxilliary result in a linear algebra form:
	
	\begin{proposition}
		\label{finiteprimematrix}
		Let $\vA$ and $\{\va_i: i \in I\}$ be $D$-dimensional vectors over $\Z$. Consider the matrix
		\begin{equation*}
			W = \begin{pmatrix}
				\longleftarrow \vA^T \longrightarrow \\
				\vdots \\
				\longleftarrow \va_{i}^T \longrightarrow \\
				\vdots 
			\end{pmatrix}
		\end{equation*}
		where the number of rows is some $r \le D$. Let the columns of $W$ be $W_1, \dots, W_D$, and suppose that for any $r$ distinct integers $1 \le d_1, \dots, d_r \le D$, $\det (W_{d_1, \dots, d_r}) \neq 0$. Then, for all but finitely many primes $p$, $\det (W_{d_1, \dots, d_r}) \neq 0$ in $\F_p$ for any choice of $r$ distinct integers $1 \le d_1, \dots, d_r \le D$.
	\end{proposition}

	\vspace{0.5cm}
	\subsection{Polynomial Root Sharing}
	
	The core idea in proving our main results is viewing equations as polynomials evaluated at some input. This requires studying the root-sharing properties of (multivariate) linear polynomials.
	
	\begin{lemma}
		\label{commonroot}
		Given $m$ linear polynomials of integer coefficients in $n$ variables, if they share a common root in $\F_p$ for infinitely many primes $p$, then they share a common root in $\Q$.
	\end{lemma}
	
	\begin{proof}
		Suppose there are $m$ linear polynomials $f_1, \dots, f_m \in \Q[x_1, \dots, x_n]$. 
		Construct a system of linear equations by $\{f_i = 0 : i \in [m]\}$, and notice that it has no solution. 
		This means one can find scalars $c_1, \dots, c_m$ such that 
		\begin{equation*}
			\sum_{i=1}^{m} c_i f_i = c \neq 0 \in \Q \,.
		\end{equation*}
		
		However, the system of equations should have a solution over $\F_p$ for infinitely many prime $p$. 
		This means $c = 0$ for infinitely many prime $p$, which indicates that $c = 0$ in $\Q$, contradiction. 
		
		Thus the system of linear equations has a solution over $\Q$. 
		Hence the linear polynomials $f_1, \dots, f_m$ share a common rational root. 
	\end{proof}

	Suppose we know that there exists some semi-monochromatic solutions under some finite colorings. This gives us a set of equations. But to prove the necessity of the $k$-columns conditions, we retreat a bit, view them as linear polynomials evaluated at some inputs, and analyze the polynomials themselves. This allows us to apply some algebraic results like Lemma \ref{commonroot}.

	\vspace{1cm}
	\section{Proof of Results (1): Sufficiency of Columns Conditions}
	\label{sufficiency}

	\subsection{Intuition from 1-Dimensional Equation}
	
	Let us first consider the $k=2$ case.
	
	Recall that the 2-columns condition indicates that there exists some subsets of indices $I_1 \subset [N]$ and $J_1 \subset [M]$, with $\va_i$'s and some $\vb_j$'s for $i \in I_1$ and $j \in J_1$, such that $\sum_{i \in I_1} \va_i$ and $\sum_{j \in J_1} \vb_j$ are parallel. In 1-dimension case, suppose the bipartite equation appears as follows:
	
	\begin{equation}
		\label{1dimeq}
		\sum_{i=1}^{N} a_i x_i + \sum_{j=1}^{M} b_j y_j = 0.
	\end{equation}
	Then the 2-columns condition is satisfied if there exist some coefficients $a_{i_0}, b_{j_0}$ and not-both-zero weights $\gamma, \delta$ such that $a_{i_0} \gamma + b_{j_0} \delta = 0$.

	\subsection{Sufficiency on Multi-Dimensional System}
	
	Suppose $S$ is a $(D, k)$-system. Assume it satisfies the $k$-columns condition, so we may find not-all-zero integers $\delta_1, \dots, \delta_k$ such that the equation
	\begin{equation*}
		\sum_{j = 1}^{k} \sum_{i=1}^{N_j} (\delta_j \va_{j,i}) x_{j,i}' = 0
	\end{equation*}
	is a $D$-dimensional linear system that satisfies the columns condition (original version in Definition \ref{col}). Assume WLOG that $\delta_1, \dots, \delta_{k'}$ are nonzero for some $k' \le k$, and $\delta_{k'+1} = \cdots = \delta_k = 0$. To apply Rado's theorem, let's rewrite the above equation as:
	\begin{equation}
		\label{kifpart1c}
		\sum_{j = 1}^{k'} \sum_{i=1}^{N_j} (\delta_j \va_{j,i}) x_{j,i}' = 0
	\end{equation}
	
	Suppose $\Z$ is $r$-colored by $\chi$, then we may induce an $r^{k'}$-coloring $\chi^\ast$ of $\Z_+$ via
	$$\chi^\ast(z) = (\chi(\delta_j z): 1 \le j \le k').$$
	By Theorem \ref{rado}, the system of equations in (\ref{kifpart1c}) admits a monochromatic solution  
	\begin{equation*}
		x_{j,i}' = z_{j,i}', \hspace{1cm} 1\le j \le k', 1 \le i \le N_j
	\end{equation*}
	in $\Z_+$ (under $\chi^\ast$). By definition of $\chi^\ast$, we know for each $j \in \{1, \cdots, k'\}$, the set $\{\delta_j z_{j,i}': 1 \le i \le N_j\} \subset \Z$ is monochromatic under $\chi$. Take $x_{j, i} = \delta_j x_{j,i}'$ for every $j \in [k']$ and $i \in [N_j]$. Take $x_{j,i} =0$ for $j > k'$ if such $j$'s exist. This gives a semi-monochromatic solution to the $k$-partite linear system $S$, and since not all $\delta_j$'s are 0, this is a nontrivial solution.
	
	In the next section, we will see what happens if a $(D,k)$-system is known to be semi-regular.

	\vspace{1cm}
	\section{Proof of Results (2): Necessity of $k$-Columns Condition}
	\label{necessity}

	\subsection{A Special Type of Colorings}
	
	We first introduce the \emph{smod $p$ coloring} defined in \cite{grs90}.
	
	For a nonzero rational number $q$ and a prime number $p$, we may uniquely write $q = p^m \cdot \dfrac{a}{b}$, where $a, b, m \in \Z$ s.t. $b > 0$, $\gcd(a, b) = 1$, and $a, b \notin p\Z$. Notice that $ab^{-1}$ is well-defined in $\F_p = \Z/p\Z$. The smod $p$ coloring is a map from $\Q^\ast$ to $\F_p$ via $q \mapsto a b^{-1}$. This is a $(p-1)$-coloring of $\Q^\ast$, and in particular, $\Z^\ast = \Z \setminus \{0\}$. We denote it by $\sigma_p(q) = a b^{-1} \pmod p$. We may extend this coloring via defining $\sigma_p(0) = 0$.
	
	Notice that the $m$ is unique to $q$. When $q \in \Z$, we call this the \emph{rank} of $q$ modulo $p$ and denote it as $o_p(n)$ (the $p$-adic order). By convention, $o_p(0) = \infty$ for any prime $p$.
	
	We will work on the smod $p$ colorings of $\Z$ to justify the necessity of the $k$-columns conditions.

	\vspace{0.5cm}
	\subsection{Associating Primes to Solutions}
	
	With the ``rank'' defined above, we may partition a solution to the linear system according to the ranks of the numbers in the solutions.
	
	Fix a semi-regular $(D,k)$-system $S$. For each prime $p$, there is a non-trivial semi-monochromatic solution to $S$ under the smod $p$ coloring $\sigma_p$, namely 
	\begin{equation}
		\label{eq:semisol}
			x_{j,i} = z_{j,i}, j \in [k], i \in [N_j]
	\end{equation}
	Then $\sigma_p(z_{j,1}) = \cdots = \sigma_p(z_{j,N_j})$ for all $j$'s. Look at the ranks of $z_{j,i}$'s smod $p$, and partition the \emph{index set} according to the ranks into
	\begin{equation}
		\label{eq:rankpart}
		[N_j] = \bigsqcup_{s = 1}^t I_{j,s},
	\end{equation}
	For each $s \in [t]$, we have for every $j \in [k]$ and every $i, i' \in N_j$ that $o_p(z_{j,i}) = o_p(z_{j,i'})$, and denote this number by $m_{p,s}$. WLOG we require $0 \le m_{p,1} < \cdots < m_{p,t}$. Define a map RP from the set of prime numbers to the set of all partitions of the $[N_j]$'s. Denote
	$$\RP(p) = (I_{j,s})_{j \in [k], s \in [t]}.$$
	
	Given a semi-regular linear system, its semi-monochromatic solution might not be unique, but it suffices to pick an arbitrary such solution for the above association. There are infinitely many primes, while there are only finitely many ways to partion the index set of the coefficients of a fixed system $S$. By Pigeon-Hole Principle, there exists an infinite subset $P_0$ of primes such that $\RP(p)$ is identical for every $p \in P_0$. In other word, there exists an infinite subset of primes $P_0$ and a rank-partition $(I_{j,s})_{j \in [k], s \in [t]}$ such that $\RP(p) = (I_{j,s})_{j \in [k], s \in [t]}$ for all $p \in P_0$. 
	
	We will prove a lemma that shows how we will make use of this notion of rank partition.
	\vspace{0.3cm}
	
	\begin{lemma}
		\label{rank-modulo}
		For each prime $p \in P_0$, let $\alpha_p = (\alpha_{p, 1}, \dots, \alpha_{p, k})$ be the colors of the semi-monochromatic solution (under $\sigma_p$) to $S$. 
		For each $s \in [t]$, let $m_{p,s} = o_p(z_{j,i})$ for $j \in [k]$ and $i \in I_{j,s}$. By definition, $0 \le m_{p,1} < \cdots < m_{p,t}$. Note that this sequence of ranks is not necessarily identical for each of the primes $p \in P_0$.
		For each $s \in [t]$ and $j \in [k]$, write $$\vA_{j, s} = \sum_{i \in I_{j,s}} \va_{j, i}.$$
		
		Then for infinite many primes $p \in P_0$, there exists some $j_\ast \in [k]$ and $\beta_{p,j}$ for every $j \in [k] \setminus \{j_\ast\}$ such that
		\begin{equation}
			\label{level1star}
			\vA_{j_\ast,1} + \sum_{j \neq j_\ast} \beta_{p,j} \vA_{j,1} \equiv \vzero \pmod p,
		\end{equation}
		and for every $s > 1$,
		\begin{equation}
			\label{leveltaustar}
			\alpha_{p, j_\ast}^{-1} \sum_{s' < s} \sum_{j=1}^{k} \sum_{i \in I_{j,s'}} x_{j, i} \va_{j, i} = p^{m_{p, s}} \left( \vA_{j_\ast, s} + \sum_{j \neq j_\ast} \beta_{p,j} \vA_{j,s}  \right) \pmod {p^{m_{p, s} + 1}}.
		\end{equation}
	\end{lemma}

	\begin{proof}[Proof of Lemma \ref{rank-modulo}]
		First look at our original system $S$. It admits a semi-monochromatic solution $(z_{j,i})$, so 
		\begin{equation}
			\sum_{j \in [k]} \sum_{i \in [N_j]} z_{j,i} \va_{j, i} = \vzero.
		\end{equation}
		Taking modulo $p^{m_{p,1}+1}$ gives
		\begin{equation}
			\label{level1alpha}
			\sum_{j = 1}^{k} \alpha_{p, j} \vA_{j, 1} = \vzero \pmod p,
		\end{equation}
		and for every $s$ such that $1 < s \le t$, similarly we have
		\begin{equation}
			\label{leveltaualpha}
			\sum_{s' < s} \sum_{j=1}^{k} \sum_{i \in I_{j,s'}} x_{j, i} \va_{j, i} = p^{m_{p, s}} \left( \sum_{j=1}^{k} \alpha_{p, j} \vA_{j, s} \right) \pmod {p^{m_{p, s} + 1}}
		\end{equation}
		where $0 \le m_{p, 1} < \cdots < m_{p, t}$ are the exponents of each rank.
		
		Notice that there must exists some $j_\ast$ such that $\alpha_{p,j_\ast} \neq 0$ in $\F_p$ for infinitely many $p$'s. If not, there would be infinitely many $p$'s giving $\alpha_p = (0, \dots, 0)$, which indicates the trivial solution. 
		
		This allows us to define $\alpha_{p, j_\ast}^{-1} \in \Z$, where $\alpha_{p, j_\ast} \alpha_{p, j_\ast}^{-1} \equiv 1 \pmod p$, for infinitely many primes $p \in P_0$. Denote this infinite subset of primes by $P_1 \subset P_0$. Let $\beta_{p,j} = \alpha_{p, j_\ast}^{-1} \alpha_{p, j}$ for each $j \neq j_\ast$. Multiplying both sides of (\ref{level1alpha}) and (\ref{leveltaualpha}) gives (\ref{level1star}) and (\ref{leveltaustar}).
	\end{proof}

	We may assume WLOG that $j_\ast = 1$ in Lemma \ref{rank-modulo}. Then the above equations become
	\begin{equation}
		\label{level1}
		\vA_{1, 1} + \beta_{p, 2} \vA_{2, 1} + \cdots + \beta_{p, k} \vA_{k, 1} = \vzero \pmod p
	\end{equation}
	and for every $s > 1$, 
	\begin{equation}
		\label{leveltau}
		\alpha_{p, 1}^{-1} \sum_{s' < s} \sum_{j=1}^{k} \sum_{i \in I_{j,s'}} x_{j, i} \va_{j, i} = p^{m_{p, s}} \left( \vA_{1, s} + \beta_{p, 2} \vA_{2, s} + \cdots + \beta_{p, k} \vA_{k, s} \right) \pmod {p^{m_{p, s} + 1}}
	\end{equation}

	\subsection{Construction of Polynomials}
	
	Observe that each coordinate of (\ref{level1}) is a congruence equation in $\F_p$, from which we get $D$-many polynomials with coefficients in $\Z$ that share a root in $\F_p^{k-1}$. Explicitly, consider any $d$ such that $1 \le d \le D$, we know
	$$\vA_{1, 1}^{(d)} + \beta_{p, 2} \vA_{2, 1}^{(d)} + \cdots + \beta_{p, k} \vA_{k, 1}^{(d)} = 0 \pmod p.$$
	This means the polynomial $f_d(z_2, \dots, z_k) = \vA_{1, 1}^{(d)} + z_2 \vA_{2, 1}^{(d)} + \cdots + z_k \vA_{k, 1}^{(d)}$ has a root $(z_2, \dots, z_k) = (\beta_{p, 2}, \dots, \beta_{p, k})$ modulo $p$.
	
	\begin{remark}
		Note that if the $f_d$'s are all zero polynomials, then we know $\vA_{j,1} = \vzero$ for every $j$.
	\end{remark}

	\vspace{0.5cm}
	For each $s > 1$, look at (\ref{leveltau}). Extract a maximal $\Q$-linearly independent subset (basis of the span) $\B_s = \{e_\iota: \iota \in \mathcal{I}_s\}$ of 
	$$\bigcup_{s' < s} \bigcup_{j=1}^k \{\va_{j,i}: i \in I_{j,s'}\}.$$
	It is guaranteed that $\B_s$ is non-empty. Then consider the matrix-valued function:
	\begin{equation}
		\label{matrixWs}
		W_{s} (z_2, \dots, z_k) = \begin{pmatrix}
			\vA_{1, s}^T + z_2 \vA_{2, s}^T + \cdots + z_k \vA_{k, s}^T \\
			\vdots \\
			\longleftarrow e_\iota^T \longrightarrow \\
			\vdots 
		\end{pmatrix}_{\iota \in \mathcal{I}}.
	\end{equation}

	Now we want to construct polynomials from the determinants of maximal submatrices of $W_s$.
	
	\begin{definition}
		Suppose $W$ is an $R \times D$ matrix, where $R \le D$. A \emph{maximal submatrix} of $W$ is a square matrix consisting of $R$-columns of $W$. Explicitly,a maximal submatrix is  
		$$W_{c_1, \dots, c_R}$$
		for some $1 \le c_1 < \cdots < c_R \le D$.
	\end{definition}

	\begin{lemma}
		\label{matrix2polynomial}
		Let $W_s$ be defined as in (\ref{matrixWs}) and $\B_s$ be the maximal $\Q$-linearly independent subset. Suppose $D \ge 1 + \abs{\B_s}$. Then there exist infinitely many primes $p \in P_1$ such that, for every $1 \le c_1 < \cdots < c_R \le D$, 
		$$\det ((W_s(z_2, \dots, z_k))_{c_1, \dots, c_R})$$
		is either a zero polynomial or an inhomogeneous polynomial in $(z_2, \dots, z_k)$ with root $(\beta_{p, 2}, \dots, \beta_{p, k})$ in $\F_p^{k-1}$.
	\end{lemma}

	\begin{proof}[Proof of Lemma \ref{matrix2polynomial}]
		First, for every $p \in P_1$, the rows of $W_s (\beta_{p, 2}, \dots, \beta_{p, k})$ are linearly dependent as elements in a $\Z / p^{m_{p, s} + 1} \Z$-module, so linearly dependent as elements in an $\F_p$-vector space. Then
		\begin{equation}
			\det ((W_s(\beta_{p, 2}, \dots, \beta_{p, k}))_{c_1, \dots, c_R}) = 0 \in \F_p.
		\end{equation}
	
		To avoid just sharing the trivial solution in $\F_p$'s, we expect the polynomials $$\det ((W_s(z_2, \dots, z_k))_{c_1, \dots, c_R})$$ to be inhomogeneous. For linear polynomials, this is equivalent to having nonzero constant term. One way to do this is to exclude primes such that $\B_s \cup \{\vA_{1, s}\}$ is linearly dependent over $\F_p$.
		
		If for all $s$, the set $\B_s \cup \{\vA_{1, s}\}$ is linearly dependent over $\Q$, then we always have $\vA_{1,s} \in span(\B_s)$. In this case, $\vA_{1,s}$ is already a linear combination of the vectors in $$\bigcup_{s' < s} \bigcup_{j=1}^k \{\va_{j,i}: i \in I^j_{s'}\}$$
		for all $s > 1$, and we simply ignore all the polynomials.
		
		When in the case that for some $s$, the set $\B_s \cup \{\vA_{1, s}\}$ is linear independent over $\Q$, then with Proposition \ref{finiteprimematrix}, we know for only finitely many primes $p \in P_1$, $\B_s \cup \{\vA_{1, s}\}$ is linearly dependent over $\F_p$. Exclude those primes from $P_1$, and we still get an infinite subset $P_2 \subset P_0$. For every $p \in P_2$ and the $s$ such that $\B_s \cup \{\vA_{1, s}\}$ is linear independent over $\Q$, we have
		$$\det ((W_s(z_2, \dots, z_k))_{c_1, \dots, c_R})$$
		inhomogenous.
	\end{proof}

	\begin{remark}
		In Lemma \ref{matrix2polynomial}, we assume $R \le D$. The opposite case is trivial, wherein we see the number of rows is strictly less than the number of columns. This means the columns must be linearly dependent, regardless of the field the vector space is defined on.
	\end{remark}

	Now, the polynomials generated by the determinants, along with those from equation (\ref{level1}), all share the common root $(\beta_{p, 2}, \dots, \beta_{p, k})$ over the field $\F_p$. Notice that they are all linear polynomials. Denote the collection of polynomials by $G$. By virtue of Lemma~\ref{commonroot}, there exists some $\gamma_2, \dots, \gamma_k \in \Q$ such that for each $g \in G$, $g(\gamma_2, \dots, \gamma_k) = 0$.
	
	\begin{remark}
		The following must hold simultaneously to prevent us from getting any polynomials:
		\begin{itemize}
			\item every $\vA_{j,1} = \vzero$,
			
			\item for each $s > 1$, $D < 1 + \abs{\B_s}$ or $\B_s \cup \{\vA_{1,s}\}$ is $\Q$-linearly.
		\end{itemize}
		Observe that, when every $\vA_{j,1} = \vzero$, the choice of $\delta_j$'s has nothing to do with the 1st rank.
		
		For $s > 1$, if $\B_s \cup \{\vA_{1,s}\}$ is $\Q$-linearly, we know $\vA_{1,s} \in span(\B_s)$. The other case is a bit tricky. Notice that $\B_s$ is a linearly independent subset of
		$$\bigcup_{s' < s} \bigcup_{j=1}^k \{\va_{j,i}: i \in I_{j,s'}\},$$
		which consists of vectors in $\Q^D$. Any linearly independent subset would not exceed a size of $D$, so $\abs{\B_s} \le D$. If $D < 1 + \abs{\B_s}$, then the only possible case is $\abs{\B_s} = D$. This means the set $\B_s$ contains $D$-many $\Q$-linearly independent vectors, so $span(\B_s) = \Q^D$. This gives $\vA_{1,s} \in span(\B_s)$ as well.
		
		With the above two requirements hold, we may set $\delta_1 = 1$ and $\delta_2 = \cdots = \delta_k = 0$ and complete the proof early.
	\end{remark}

	\subsection{Polynomial Implications}
	
	If at the end we get no polynomials from the above process, by the remark we know that
	
	\begin{itemize}
		\item $\vA_{1, 1} = \vzero$;
		
		\item For each $s > 1$, $\vA_{1, s}$ is a linear combination of $$\bigcup_{s' < s} \bigcup_{j=1}^k \{\va_{j,i}: i \in I_{j,s'}\}.$$
	\end{itemize} 
	
	In this case, simply choosing $\delta_1 = 1$ and $\delta_2 = \cdots = \delta_k = 0$ verifies the $k$-Columns Condition.
	
	\vspace{0.3cm}
	On the other hand, suppose the set of polynomials $G$ in non-empty. They share a common root $(\gamma_2, \dots, \gamma_k)$ over $\Q$. This means
	\begin{equation*}
		\vA_{1, 1} + \gamma_2 \vA_{2, 1} + \cdots + \gamma_k \vA_{k, 1} = \vzero \in \Q^D
	\end{equation*}
	and for each $s > 1$, the rows of $W_s (\gamma_2, \dots, \gamma_k)$ are linearly dependent over $\Q$. That is, each $$\vA_{1, s} + \gamma_2 \vA_{2, s} + \cdots + \gamma_k \vA_{k, s}$$ is a linear combination of $\B_s$ over $\Q$, where $\B_s$ is simply a subset of $$\bigcup_{s' < s} \bigcup_{j=1}^k \{\va_{j,i}: i \in I_{j,s'}\}.$$
	
	By assumption, at least some of the polynomials in $G$ are inhomogeneous, which means the shared rational solution $\gamma$ is not trivial. Let $\delta_1$ be some large enough common multiple of the nonzero denominators of $\gamma_j$, $2 \le j \le k$, and $\delta_j = \gamma_j \delta_1$ for each $2 \le j \le k$. Then for each $1 \le s \le t$, $\sum_{j=1}^k \delta_j \vA_{j, s}$ is a $\Z$-linear combination of vectors in
	$$\bigcup_{s' < s} \bigcup_{j=1}^k \{\va_{j, i}: i \in I_{j,s'}\}.$$
	This ultimately gives the $k$-columns condition.

	\vspace{1cm}
	\section{Discussions}

	\subsection{Computational Meanings}
	
	The \emph{Compactness Principle} states that, given any regular system $S$, there exists some sufficiently large integer $R$ such that when $[R]$ is  $r$-colored, there is a monochromatic solution to $S$ within $[R]$. We call the least such number $R$ the \emph{Rado number}, denoted by $\rado(S, r)$. There is no general big-O bound on $\rado(S, r)$ presented in \cite{grs90}, but the numbers on specific equations have been studied (see e.g. \cite{foxkleit}, \cite{radobound}, and \cite{ newradobound}). 
	
	Observe that our $k$-columns condition is weaker than the columns condition. It should be ``easier'' to find a coefficient matrix that satisfies a weaker condition in the sense that the searching range for the entries can be smaller. When we require only the semi-regularity of a system of linear equations, it should take shorter time to actually find out a semi-monochromatic solution.
	
	\subsection{Conjectures Beyond the Problem}
	
	From \cite{canonical} we know that the Rado's Theorem (Theorem \ref{rado}) can be generalized to arbitrary abelian groups. Notice that any ring has an underlying abelian group, so we may consider the multiplication by coefficients as the ring multiplication. We want to see if our results also apply in other rings. 
	
	Recall Definition \ref{dk-system} and Definition \ref{kcol}, where we only considered equations with integer coefficients. We may extend these definitions to the field of fractions, i.e., the rationals. Suppose a matrix $M \in M_{d \times n}(\Q)$ is the coefficient matrix of a $(D,k)$-system (in the extended sense). Construct a matrix $M' \in M_{d \times n} (\Z)$ simply by multiplying $M$ by the LCM of all denominators of its nonzero coefficients. The matrix $M'$ satisfies the $k$-columns conditions (in the extended sense) if and only if $M$ satisfies the $k$-columns conditions (in the original sense). 
	
	This generalization applies to any arbitrary ring and its field of fractions. For the conjecture below, we will consider linear systems with coefficients in some bigger fields. Before that, we need to introduce some number theory concepts.

	\begin{definition}
		\label{algint}
		A number $\alpha \in \C$ is called an \emph{algebraic integers} if it is a root of some monic polynomial in $\Z[x]$. The set of all algebraic integers in $\C$ forms a ring and is denoted by $\bar{\Z}$.
	\end{definition}
	
	\begin{definition}
		\label{numfield}
		$K$ is called an \emph{algebraic number field} (or simply \emph{number field}) if $K \subset \C$ is a field such that $[K:\Q] < \infty$. Its \emph{ring of integers} is the set $\roi_K = K \cap \bar{\Z}$. Note that $\roi_K$ is a ring.
	\end{definition}
	
	Now we may identify $\Z$ as $\Q \cap \bar{\Z}$, the ring of integers of $\Q$. Recall the smod $p$ coloring in Section 4.1, and we want to find an analogue of such coloring in $\roi_K$. Let $a \in \roi_K$ and $p \in \roi_K$ be some prime element. Then $(p)$ is a prime ideal. 
	
	Asssume further that $\roi_K$ is a PID. Then it is a UFD, so any $a \in \roi_K \setminus \{0\}$ has a unique factorization that gives $a = p^m b$ for some unique $m \in \N$ and $b \in \roi_K$ coprime with $p$. Define the color of $a$ as the coset $b + (p)$. Also, define the color of $0$ to be the coset $0 + (p) = (p)$. Since $\roi_K / (p)$ is a finite field, this is a well-defined finite coloring of $\roi_K$.
	
	With this coloring, one can work on the modulo equations and polynomials over $K$ and $\roi_K / (p)$ with prime element $p \in \roi_K$. We extrapolate that a procedure similar to Section 4 should lead to the following result.
	
	\begin{conjecture}
		\label{rado-roi}
		Let $K \subset \C$ be a number field. Assume further that its ring of integers $\roi_K$ is a PID. Now suppose $S$ is a $(D, k)$-system
		\begin{equation}
			\sum_{j=1}^{k} \sum_{i=1}^{N_j} \va_{j, i} x_{j,i} = \vzero
		\end{equation}
		with coefficients $\va_{j, i}^{(d)}$ in $K$.
		
		We call $S$ is semi-regular if it admits a semi-monochromatic solution in $\roi_K$ under any finite coloring of $\roi_K$. Then $S$ is semi-regular if and only if it satisfies the $k$-columns condition.
	\end{conjecture}
	
	\begin{remark}
		In the definition of smod $p$ coloring, we only make use of the unique factorization property. So, technically speaking, assuming that $\roi_K$ is UFD would be enough. But it is easier to check whether $\roi_K$ is a PID with the Minkowski's bound.
	\end{remark}

\section*{Acknowledgements}

I would like to thank Professor Boris Bukh for his guidance and mentorship throughout this project. I deeply appreciate the generosity with which he has given me his time and energy. This project was funded by Undergraduate Research Office, Carnegie Mellon University.

	\bibliography{reference.bib}

\begin{thebibliography}{GMT12}

\bibitem[Chi90]{chionh90}
Eng-Wee Chionh.
\newblock {\em Base Points, Resultants, and the Implicit Representation of
  Rational Surfaces}.
\newblock PhD thesis, University of Waterloo, 1990.

\bibitem[DF04]{dumfoo}
David~S. Dummit and Richard~M. Foote.
\newblock {\em Abstract Algebra Third Edition}.
\newblock John Wiley \& Sons, Inc., New York, NY, 2004.

\bibitem[DL89]{canonical}
Walter~A. Deuber and Hanno Lefmann.
\newblock Partition regular systems of homogeneous linear equations over
  abelian groups: the canonical case.
\newblock {\em Combinatorial Mathematics, Proceedings of the Third
  International Conference}, 555:167--170, 1989.

\bibitem[Du92]{du92}
Hang~Khanh Du.
\newblock {\em New Resolvent Methods with Applications to Curves and Surfaces
  in Geometric Modeling}.
\newblock PhD thesis, University of Waterloo, 1992.

\bibitem[FK06]{foxkleit}
Jacob Fox and Daniel~J. Kleitman.
\newblock On rado's boundedness conjecture.
\newblock {\em Journal of Combinatorial Theory Series A}, 113:84--100, 2006.

\bibitem[GMT12]{newradobound}
William Gasarch, Russel Moriarty, and Nithin Tumma.
\newblock New upper and lower bounds on the rado numbers, 2012.

\bibitem[Gol06]{goldman}
Ron Goldman.
\newblock {\em Algebraic Geometry and Geometric Modeling}.
\newblock Springer-Verlag, 2006.

\bibitem[GRS90]{grs90}
Ronald~L. Graham, Bruce~L. Rothschild, and Joel~H. Spencer.
\newblock {\em Ramsey Theory}.
\newblock John Wiley \& Sons, Inc., New York, NY, 1990.

\bibitem[Mor11]{radobound}
Russell Moriarty.
\newblock Better bounds on rado numbers, 2011.

\bibitem[Rad33]{rado33}
Richard Rado.
\newblock Studien zur kombinatorik.
\newblock {\em Mathematische Zeitschrift}, 36:424--470, 1933.

\bibitem[vdW50]{algvan}
B.~L. van~der Waerden.
\newblock {\em Modern Algebra}.
\newblock Frederick Ungar Publishing Co., New York, NY, 1950.

\end{thebibliography}
	\bibliographystyle{alpha}
	\vspace{0.5cm}

\end{document}